\documentclass[11pt,twoside,a4paper]{article}
\usepackage{myartmods}
\usepackage{macros}
\usepackage{cite}
\usepackage{datetime}\usdate
\usepackage[utf8]{inputenc}
\usepackage[english]{babel}
\usepackage{amsmath}
\usepackage{verbatim}
\usepackage{amsfonts}
\usepackage{color}
\usepackage{algorithm}
\usepackage{algpseudocode} 

\input{macros}

\algrenewcommand{\algorithmiccomment}[1]{\hfill[#1]}

\title{On solving symmetric systems of linear
  equations in an unnormalized Krylov subspace
  framework\thanks{\footsupersede}} 
\author{Anders FORSGREN\thanks{\footTO} \and Tove
  ODLAND\addtocounter{footnote}{-1}\footnotemark}

\def\footTO{Optimization and Systems Theory, Department of
  Mathematics, KTH Royal Institute of Technology, SE-100 44 Stockholm,
  Sweden ({\tt andersf@kth.se,odland@kth.se}).}

\def\footsupersede{This manuscript
  supersedes Report TRITA-MAT-2014-OS-01 {\em ``A general Krylov
    method for solving symmetric systems of linear equations''},
  Department of Mathematics, KTH Royal Institute of Technology, March 2014.}

\date {September 17, 2014}

\begin{document}
\maketitle\thispagestyle{empty}

\begin{abstract}
  In an unnormalized Krylov subspace framework for solving symmetric
  systems of linear equations, the orthogonal vectors
  that are generated by a Lanczos process are not necessarily on the
  form of gradients. Associating each orthogonal
  vector with a triple, and using only the three-term recurrences of
  the triples, we give conditions on whether a
  symmetric system of linear equations is compatible or incompatible.
  In the compatible case, a
  solution is given and in the incompatible case, a certificate of
  incompatibility is obtained. In particular, the case when the matrix is
  singular is handled. 

We also derive a minimum-residual method based on this
  framework and show how the iterates may be updated explicitly based
  on the triples, and in the incompatible case a minimum-residual solution of minimum
  Euclidean norm is obtained.

\smallskip


\smallskip
\noindent{\bf MCS number:} 65F10
\end{abstract}

\section{Introduction}

An important problem in numerical linear algebra and optimization is to solve a system
of equations where the matrix is symmetric. Such a problem may be posed as
\begin{equation}\label{lineq}
Hx+c=0,
\end{equation}
for $x \in \mathbb{R}^n$, with $c \in \mathbb{R}^n$ and $H=H^T \in
\mathbb{R}^{n\times n}$. Note that with $A=H$ and $b=-c$,
\eqref{lineq} becomes $Ax=b$. However, we prefer the notation of \eqref{lineq} as it is on the
form of a gradient $g$, defined as $g=Hx+c$, being equal to zero. This
notation highlights that we are trying to find a
non-trivial linear combination of the columns of $H$ and $c$. Our primary
motivation comes from optimization where in many cases the systems of linear equations
that need to be solved are such that the matrix $H$ is symmetric but in general
indefinite. For example, KKT systems have this
form, see, e.g., \cite{FGS96}, but there are many other
applications. Throughout, $H$ is assumed symmetric, any other
assumptions on $H$ at particular instances will be stated
explicitly. The key concept in this paper will be to determine if \eqref{lineq} is
compatible or not. Our results include and handle the case when $H$ is a
singular matrix. It is assumed throughout that $c\ne 0$.

Exact arithmetic will be assumed and the theory developed in this paper
is based on that assumption. In the end of the paper we briefly
discuss computational aspects of our results in finite precision.

One strategy for solving \eqref{lineq} is to
generate linearly independent vectors, $q_k$, $k=0, 1, \dots$ until
$q_k$ becomes linearly dependent for some $k=r \leq n$ and hence
$q_r=0$. In
this paper we consider Krylov subspace methods in which the
generated vectors form an orthogonal, hence linearly independent, basis for the Krylov subspaces
generated by $H$ and $c$,
\begin{equation}\label{eqn-Krylov}
\mathcal{K}_0(c, H)=\{0\}, \quad \mathcal{K}_k(c, H)=span\{c, Hc,
H^2c, \dots, H^{k-1}c\}, \quad k=1, 2, \dots.
\end{equation}
The Krylov vectors $c, Hc, \dots, H^{r-1}c$ are linearly independent, but
as they become highly ill-conditioned it is desirable to work with
some other set of vectors.

Orthogonal vectors $q_k$ that are generated by a
Lanczos process will be a linear combination of the
columns of $H$ and $c$. There is a freedom in the scaling of each generated
vector. We will refer to the case when the coefficient
corresponding to $c$ is equal to one as a \emph{normalized} Lanczos
vector, i.e the vector is on the form of a gradient, $g=Hx+c$. 
An \emph{unnormalized} Lanczos vector is then referring to the case
when the coefficient corresponding to $c$ is not required to be one,
i.e. $q=Hy+\delta c$, where $\delta \in\Re$.

The concept of using unnormalized 
Lanczos vectors was introduced by Gutknecht in \cite{Gutknecht90, Gutknecht92} as a
remedy for so called pivot breakdowns that occur when normalization is
not well defined.\footnote{Gutknecht considers the more general case
  when $H$ is non symmetric where there are several
  other possible breakdowns for corresponding Lanczos a process, see, e.g., \cite{saad, gutknechtacta}. For $H$
  symmetric, the pivot breakdown is the only one that can happen.} In subsequent
work by Gutknecht the term \emph{inconsistent} is used, see, 
e.g., \cite{gutknechtacta, Gutknecht00}. However, in this
paper the term unnormalized will be used as it better suits our
purposes. The unnormalized framework will be used in a more general sense, not only
as a remedy for pivot breakdown, to derive our results.

The Lanczos process was first introduced by Lanczos
\cite{lanczos50,lanczos52}. There have been very many contributions to
the theory both for symmetric and non-symmetric systems, 
see, e.g., Golub and O'Leary's extensive survey of the years
1948-1976 \cite{goluboleary}, Golub and Van Loan's book
\cite{golubvanloan} and Gutknecht's survey \cite{gutknechtacta}. 

The outline of the paper is as follows. Section~\ref{sec-background}
gives a review of background material on the unnormalized Krylov
subspace framework. In particular, we review recursions for the
unnormalized Lanczos triples $(q_k, y_k, \delta_k)$ associated with
the unnormalized Lanczos vectors $q_k$, $k=0, \dots, r$, such that
$q_k=Hy_k+ \delta_kc$, $q_k \in \mathcal{K}_{k+1}(c,H)$, $y_k \in
\mathcal{K}_{k}(c,H)$ and $\delta_k \in \mathbb{R}$, $k=0, \dots, r$.

In Section~\ref{sec-prop}, we give our main convergence result, based
on the recursions for the triples, stating that when \eqref{lineq} is
compatible a solution is given (in this case we show that $\delta_r
\neq 0$), or a certificate of incompatibility can be obtained for
\eqref{lineq} (in this case $\delta_r = 0$). The case of a singular
matrix $H$ is included and handled in the analysis, which to the best of our
knowledge has not been done before.  The derivation is summarized in
an unnormalized Krylov algorithm, and in addition some remarks are
made on the case when normalization is well defined and used.

Finally, in Section \ref{sec-conminres}, a minimum-residual method,
applicable also for incompatible systems, is derived by making use of the
unnormalized Krylov framework. Explicit recursions for the
minimum-residual iterates are derived, including an expression for the
solution of minimum Euclidean norm in the incompatible case. 

\subsection{Notation}

The letter $i$, $j$ and $k$ denote integer indices, other lowercase 
letters such as $q$, $y$ and $c$ denote
column vectors, possibly with super- and$/$or subscripts.  For a
symmetric matrix $H$, $H\succ0$ denotes that $H$ is positive
definite. Analogously, $H\succeq0$ is used to denote that $H$ is positive
semidefinite. The null
space and range space of $H$ are denoted by $\mathcal{N}(H)$ and 
$\mathcal{R}(H)$ respectively. 
We will denote by $Z$ an orthonormal matrix whose
columns form a basis of $\mathcal{N}(H)$. If $H$ is nonsingular,
then $Z$ is to be interpreted as an empty matrix. When referring to a
norm, the Euclidean norm is used throughout.

\section{Background}\label{sec-background}

Regarding \eqref{lineq}, the raw data available is the matrix $H$ and
the vector $c$ and combinations of the two, for example represented by
the Krylov subspaces generated by $H$ and $c$, as defined in
\eqref{eqn-Krylov}. For an introduction and background on Krylov
subspaces, see, e.g., Gutknecht~\cite{Gutknecht07} and Saad~\cite{saad}.

Without loss of generality, the scaling of the first vector
$q_0$ may be chosen so that $q_0=c$. Then one sequence of linearly
independent vectors may be generated by letting $q_k \in
\mathcal{K}_{k+1}(c, H)\cap \mathcal{K}_k(c, H)^{\perp}$, $ k=1,
\dots, r$, such that $q_k \neq 0$, for $k=0, 1, \dots, r-1$ and
$q_r=0$ where $r$ is the minimum index $k$ for which $
\mathcal{K}_{k+1}(c, H)= \mathcal{K}_k(c, H)$.  These
vectors $\{q_0, q_1, \dots, q_{r-1}\}$ form an orthogonal, hence
linearly independent, basis of $\mathcal{K}_r(c, H)$. We will refer to
these vectors as the Lanczos vectors. With $q_0=c$, each vector $q_k$,
$k=1,\dots,r-1$, is uniquely determined up to a scaling. A vector $q_k
\in \mathcal{K}_{k+1}(c, H)$ may be expressed as
\begin{equation}\label{eqn-qk}
q_k=\sum_{j=0}^k\delta^{(j)}_kH^jc,
\end{equation}
for some parameters $\delta_k^{(j)}$, $j=0,\dots,k$, uniquely determined up to
  a nonzero scaling and $\delta_k^{(k)}\ne
 0$.
This is made precise in
Lemma~\ref{lem-qkappendix}. 

Normalized Lanczos vectors are obtained when the scaling is chosen such that
$\delta_k^{(0)}=1$, and we call this the normalization
condition.\footnote{From \eqref{eqn-qk}, one can see that the vectors $q_k$ may be represented as
  $q_k=p_k(H)c$, where $p_k$ is polynomial of degree $k$, hence
  the normalization condition may be stated as
  $p_k(0)=1$, see, e.g., Gutknecht \cite{gutknechtacta}} Since $\delta_k^{(0)}$ is
determined up to a scaling it holds that if $\delta_k^{(0)} \neq 0$
then one can rescale the vector such that the normalization condition holds,
however if $\delta_k^{(0)}=0$, then this is not possible and a pivot
breakdown occurs.

The following proposition reviews a recursion for a sequence of Lanczos
vectors where the scaling factors denoted by
$\{\theta_k\}_{k=0}^{r-1}$ are left unspecified. This recursion is a slight generalization of the symmetric Lanczos
process for generating mutually orthogonal vectors, in which the usual
choice of the scaling is such that each vector $q_k$
is chosen such that $||q_k||=1$, $k=0, \dots, r-1$. For completeness,
this proposition and its proof is included.

\begin{proposition}\label{genq}
  Let $r$ denote the smallest positive integer $k$ for which $\mathcal{K}_{k+1}(c,H)= \mathcal{K}_k(c, H)$.
Given $q_0=c\in\mathcal{K}_1(c, H)$, there exist vectors $q_k$,
$k=1,\dots,r$, such that
$$
q_k \in \mathcal{K}_{k+1}(c, H)\cap \mathcal{K}_k(c, H)^{\perp}, \quad
k=1, \dots, r,
$$
for which $q_k \neq 0$, $k=1, \dots, r-1$, and $q_r=0$.  Each such
$q_k$, $k=1,\dots,r-1$, is uniquely determined up to a scaling, and a
sequence $\{q_k\}_{k=1}^r$ may be generated as
\begin{subequations}\label{Q}
\begin{eqnarray}
q_{1}&=&\theta_0 \big(-H q_0+
\frac{q_0^THq_0}{q_0^Tq_0}q_0\big),\label{q1} \\
q_{k+1}&=&\theta_{k} \big(-H q_k+
\frac{q_k^THq_k}{q_k^Tq_k}q_k+\frac{q_{k-1}^THq_k}{q_{k-1}^Tq_{k-1}}q_{k-1}\big),\label{q2}
\quad k=1,\dots, r-1,
\end{eqnarray}
\end{subequations}
where $\theta_k$, $k=0,\dots,r-1$, are free and nonzero parameters. In
addition, it holds that
\begin{equation}
q_{k+1}^T q_{k+1} = - \theta_{k} q_{k+1}^T H q_k, \quad k=0,\dots,r-1.
\end{equation}
\end{proposition}

\begin{proof}
  Given $q_0=c$, let $k$ be an integer such that $1\le k\le
  r-1$. Assume that $q_i$, $i=0,\dots,k$, are mutually orthogonal with
  $q_i\in\mathcal{K}_{i+1}(c,H)\cap \mathcal{K}_i(c, H)^{\perp}$. Let
  $q_{k+1} \in \mathcal{K}_{k+2}(c, H)$ be expressed as
\begin{equation}\label{q}
q_{k+1}=-\theta_{k} H q_k+\sum_{i=0}^k \eta_k^{(i)} q_i, \quad k=0, \dots, r-1,
\end{equation}
In order for $q_{k+1}$ to be orthogonal to $q_i$, $i=0,\dots,k$, the
parameters $\eta_k^{(i)}$, $i=0, \dots, k$, are uniquely determined as
follows.

For $k=0$, to have $q_0^Tq_1=0$, it must hold that
$$
\eta_0^{(0)}=\theta_0\frac{q_0^THq_0}{q_0^Tq_0},
$$
hence obtaining $q_{1} \in
\mathcal{K}_{2}(c, H)\cap \mathcal{K}_1(c, H)^{\perp}$ as in \eqref{q1},
where $\theta_0$ is free and nonzero.  For $k$ such that $1\le k\le
r-1$, in order to have $q_i^T q_{k+1}=0$, $i=0, \dots, k$, it must
hold that
\[
\eta_k^{(k)}  = \theta_{k}\frac{q_{k}^THq_k}{q_{k}^Tq_{k}}, \quad
\eta_k^{(k-1)}  = \theta_{k}\frac{q_{k-1}^THq_k}{q_{k-1}^Tq_{k-1}},
\text{and}
\eta_k^{(i)} =0, \quad i = 0,\dots, k-2.
\]
The last relation follows by the symmetry of $H$. Hence, obtaining
$q_{k+1} \in
\mathcal{K}_{k+2}(c, H)\cap \mathcal{K}_{k+1}(c, H)^{\perp}$ as in
the three-term recurrence of \eqref{q2},
where $\theta_{k}$, $k=1,\dots, r-1$, are free and nonzero.

Since $q_1$ is orthogonal to $q_0$, and since $q_{k+1}$ is orthogonal
to $q_k$ and $q_{k-1}$, $k=1, \dots, r-1$, pre-multiplication of \eqref{Q} with $q_{k+1}^T$ yields
\[
q_{k+1} ^T q_{k+1} =-\theta_{k} q_{k+1} ^T H q_k, \quad k=0, \dots, r-1 .
\]
Finally note that if $q_{k+1}$ is given by \eqref{Q}, then the only
term that increases the power of $H$ is $\theta_k (-Hq_k)$. Since
$\theta_k\ne0$, repeated use of this argument gives
$\delta_{k+1}^{(k+1)}\ne 0$ if $q_{k+1}$ is expressed by
\eqref{eqn-qk}. In fact,
$\delta_{k+1}^{(k+1)} = (-1)^{k+1}\prod_{i=0}^k \theta_i \ne 0$.
Hence, by Lemma~\ref{lem-qkappendix}, $q_{k+1}=0$ implies
$\mathcal{K}_{k+2}(c, H)= \mathcal{K}_{k+1}(c, H)$,
so that $k+1=r$, as required.
\end{proof}

The particular form of \eqref{Q} with scaling parameters $\theta_k$,
$k=0,\dots,r$, is made to get coherence with existing theory on the
method of conjugate gradients, see Section~\ref{sec-cg} and
Proposition~\ref{cgstep}.  To simplify the exposition, the following
notation is introduced,
\begin{equation}\label{alphabeta}
\alpha_0=\frac{q_0^THq_0}{q_0^Tq_0},
\quad\alpha_k=\frac{q_k^THq_k}{q_k^Tq_k},  \quad
\beta_{k-1}=\frac{q_{k-1}^THq_k}{q_{k-1}^Tq_{k-1}} \quad k=1, \dots, r-1.
\end{equation}

Let $Q_k$ be the matrix with the Lanczos vectors $q_0, q_1,
\dots, q_k$ as columns vectors, then \eqref{Q} may be written on
matrix form as,
$$
HQ_k=Q_{k+1}\bar{T}_k=Q_kT_k-\frac{1}{\theta_k}q_{k+1}e_{k+1}^T,
$$
where
\begin{equation}\label{matrixfac}
T_k=\left(
\begin{array}{cccc}
\alpha_0&\beta_0& &\\
-\frac{1}{\theta_0} & \ddots&\ddots & \\
& \ddots &\ddots & \beta_{k-1}\\
& & -\frac{1}{\theta_{k-1}}& \alpha_k
\end{array} \right), \quad \quad
\bar{T}_k=\left(
\begin{array}{c}
T_k\\
-\frac{1}{\theta_k}e_{k+1}^T
\end{array} \right).
\end{equation}
The choice of $\theta_k$ such that $||q_k||_2=||q_0||_2$ implies
$\beta_k=-\frac{1}{\theta_k}$ and in this case $T_k$ will be
symmetric. Changing the set of $\{\theta_k\}_{k=0}^{r-1}$ can be seen as a similarity transform
of $T_k$, see, e.g., Gutknecht \cite{gutknechtacta}.

Many methods for solving \eqref{lineq} use matrix-factorization
techniques on $T_k$ or $\bar{T}_k$. For an introduction to how Krylov subspace methods are formalized
in this way, see, e.g., Paige, Saunders and Choi \cite{paige94, choipaigesaunders}. For our purposes we leave these available scaling
factors unspecified and work with the recursions \eqref{Q} directly.

\subsection{An extended representation of the unnormalized Lanczos vectors}

To find a solution of \eqref{lineq}, if it exists, it is not
sufficient to generate the sequence $\{q_k\}_{k=1}^r$. Note that, as
in \eqref{eqn-qk}, $q_{k} \in \mathcal{K}_{k+1}(c, H)$, $k=0, \dots,
r$, may be expressed as
\begin{equation}\label{eqn-qkII}
q_k=\sum_{j=0}^k\delta^{(j)}_kH^jc=H
\big(\sum_{j=1}^k\delta^{(j)}_kH^{j-1}c\big)+\delta^{(0)}_kc, \quad k=1, \dots, r.
\end{equation}
It is not
convenient to represent $q_k$ by \eqref{eqn-qk}. Therefore, defining $y_0=0$, $\delta_0=1$,
$$
y_k =  \sum_{j=1}^k\delta^{(j)}_kH^{j-1}c \in \mathcal{K}_{k}(c, H)
\text{and} \delta_k = \delta^{(0)}_k,
\quad k=1, \dots, r,
$$ 
it follows from \eqref{eqn-qkII} that
$$
q_k=Hy_k+\delta_k c,
$$ 
so that $q_k$ may be expressed by $y_k$ and $\delta_k$. These
quantities will be represented by the triples $(q_k, y_k, \delta_k)$,
$k=0, \dots, r$.  Note that $\{\delta_k^{(j)}\}_{j=0}^k$ are given in
association with the raw data $H$ and $c$, the choice made here is to
use only $\delta_k^{(0)}$ explicitly and collect all other terms in
$y_k$.

It is
straightforward to note that if $\delta_r\ne 0$, then $0=q_r=H x_r+c$
for $x_r=(1/\delta_r)y_r$, so that $x_r$ solves \eqref{lineq}. It will
be shown that \eqref{lineq} has a solution if and only if
$\delta_r\ne 0$.

As mentioned earlier, for
a given $k$ such that $1\le k\le r$, the parameters $\delta_k^{(j)}$,
$j=1,\dots,k$, are uniquely defined up to a scaling. Hence, so is the
triple $(q_k, y_k, \delta_k)$. This is made precise in the recursions
for the triples given in Lemma~\ref{rec}. 

It is possible to use more of the coefficients $\{\delta_k^{(j)}\}_{j=1}^k$
explicitly in the same representation as above. For the next power of the polynomial in \eqref{eqn-qkII}, let
\begin{eqnarray}\label{eqn-y1}
y_k  & = & H y_k^{(1)} + \delta_k^{(1)} c, \text{with} \\
y_k^{(1)} & = & \sum_{j=2}^k\delta^{(j)}_kH^{j-2}c \in
\mathcal{K}_{k-1}(c, H),
\quad k=2, \dots, r, \nonumber
\end{eqnarray} 
in addition to $y_1^{(1)}=0$. This will be used in the
analysis, but not in the algorithm presented.

Based on Proposition \ref{genq}, given
$(q_0,y_0,\delta_0)=(c,0,1)$, one can formulate recursions for
$(q_k,y_k,\delta_k)$, $k=1, \dots, r$. 
This derivation is given by Gutknecht in e.g.
\cite{gutknechtacta}, but we give the following lemma for completeness.

\begin{lemma}\label{rec}
  Let $r$ denote the smallest positive integer $k$ for which $\mathcal{K}_{k+1}(c,
  H)= \mathcal{K}_k(c, H)$. Given $(q_0,y_0,\delta_0)=(c,0,1)$, there exist vectors
$(q_k,y_k,\delta_k)$, $k=1,\dots,r$, such that
$$
q_k \in \mathcal{K}_{k+1}(c, H)\cap \mathcal{K}_k(c, H)^{\perp}, \quad
y_k \in \mathcal{K}_k(c, H), \quad q_k=Hy_k+\delta_k c, \quad
k=1, \dots, r,
$$
for which $q_k \neq 0$, $k=1, \dots, r-1$, and $q_r=0$. Each such
$(q_k,y_k,\delta_k)$, $k=1,\dots,r$, is uniquely determined up to a
scalar, and a sequence $\{(q_k,y_k,\delta_k)\}_{k=1}^r$ may be generated as
\[
y_{1}= \theta_0\big(-q_0+ \alpha_0y_0\big),\quad
\delta_{1}= \theta_0\big(\alpha_0\delta_0\big),\quad 
q_1 =\theta_0 \big(-H q_0+
\alpha_0q_0\big), 
\]
and
\begin{align*}
y_{k+1}&=\theta_k \big(-q_k+
\alpha_ky_k+
\beta_{k-1}y_{k-1}\big),
& k = 1,\dots,r-1, \\
\delta_{k+1}&=\theta_k \big(
\alpha_k\delta_k+
\beta_{k-1}\delta_{k-1}\big),
\ & k =1,\dots, r-1, \\
q_{k+1} & =  \theta_{k} \big(-H q_k+
\alpha_kq_k+\beta_{k-1}q_{k-1}\big),
& k = 1,\dots,r-1,
\end{align*}
where $\theta_k$, $k=0,\dots,r-1$, are free and nonzero
parameters, and $\alpha_k$, $k=0,\dots,r-1$ and $\beta_{k-1}$,
$k=1,\dots,r-1$ are given by \eqref{alphabeta}. 
In addition, it holds that $y_k$ are linearly independent
for $k=1,\dots,r$.
\end{lemma}

\begin{proof}
The recursions are given by simple induction on $k$. We omit the
details, see, e.g. \cite{gutknechtacta}.

By Lemma~\ref{lem-qkappendix} it holds that for $k=0, \dots, r$,
$\delta_k^{(j)}$, $j=0, \dots, k$ are uniquely determined up to a
scaling, hence it follows that $y_{k+1}$ and $\delta_{k+1}$, $k=0,
\dots, r-1$ are uniquely determined up to a scaling by the recursions of this
proposition.

Further, note that the recursion for $y_{k+1}$ has a nonzero leading
term of $q_k$ plus additional terms of $y_i$, $i=k$ and $i=k-1$. Since
$q_k$ is orthogonal to $y_i$ for $i\le k$ and $q_k\ne 0$ for $k<r$, it
follows that the vectors $y_{k+1}$ are linearly independent for $k=0,\dots,r-1$.
\end{proof}

Note that the choice 
\begin{equation}\label{normcond}
\theta_0=\frac{1}{\alpha_0}, \quad
\theta_k=\frac{1}{\alpha_k+\beta_{k-1}}, \quad k=1, \dots, r-1,
\end{equation}
in the recursions of Lemma~\ref{rec} implies $\delta_k=1$, $k=0,
\dots, r$. Hence, this choice will give rise to Lanczos vectors
that are on the form of gradients. The terms $g_k$ and $x_k$ are reserved for this case,
and we then denote $(q_k,y_k , \delta_k)$ by $(g_k, x_k, 1)$.  Therefore, \eqref{normcond} is another way
of stating the normalization condition. Note that if
$\alpha_k+\beta_{k-1}=0$, for some $k$, then this particular choice is not well
defined and a pivot breakdown occurs.  In the unnormalized Krylov
subspace framework the choice of scaling will not be based on the value of $\delta_k$.

\section{Properties of the unnormalized Krylov
  framework}\label{sec-prop}

We will henceforth refer to the unnormalized Lanczos triples $(q_k,y_k,\delta_k)$, $k=0,
\dots, r$, as \emph{given by} Lemma~\ref{rec}. Based on the
unnormalized framework due to Gutknecht that has been described in the previous
section we will now proceed to state our results.

\subsection{Convergence in the unnormalized Krylov
  framework}\label{sec-res}
The final triple, $(q_r,y_r,\delta_r)$, can now be used to show
our main convergence result, namely that \eqref{lineq} has a solution
if and only if $\delta_r\ne 0$, and that the recursions in Lemma~\ref{rec} can be used to
find a solution if $\delta_r\ne 0$ and a certificate of
incompatibility if $\delta_r=0$. The case when $H$ is singular is
included and handled in this result. 

\begin{theorem}\label{conv}
  Let $(q_k, y_k,\delta_k)$, $k=0, \dots, r$, be given by Lemma
  \ref{rec}, and let $Z$ denote a matrix whose columns form an
  orthonormal basis for $\mathcal{N}(H)$. Then, the following holds
  for the cases $\delta_r\ne 0$ and $\delta_r=0$ respectively.
  \def\labelenumi{\alph{enumi})}
\begin{enumerate}
\item If $\delta_r \neq 0$, then $Hx_r+c=0$ for $x_r=(1/\delta_r)y_r$,
  so that $c\in\mathcal{R}(H)$ and $x_r$ solves \eqref{lineq}. In
  addition, it holds that $Z\T y_k=0$, $k=0,\dots,r$.
\item If $\delta_r = 0$, then $y_r = \delta_r^{(1)} ZZ^Tc$, with
  $\delta_r^{(1)} \neq 0$ and $Z^Tc \neq 0$, so that $c\not
  \in\mathcal{R}(H)$ and \eqref{lineq} has no solution. Further, there
  is a $y_r^{(1)}\in \mathcal{K}_{k-1}(c, H)$ so that $y_r= H
  y_r^{(1)} + \delta_r^{(1)} c$. Hence, $H(H x_r^{(1)} + c)=0$ for
  $x_r^{(1)}=(1/\delta_r^{(1)})y_r^{(1)}$, so that $x_r^{(1)}$ solves
  $\min_{x\in\Re^n}\norm{Hx+c}_2^2$.
\end{enumerate}
\end{theorem}

\begin{proof}
For (a), suppose that $\delta_r \neq 0$. Then 
$0=q_r=Hy_r+\delta_rc$, hence $Hx_r+c=0$ for $x_r=(1/\delta_r)y_r$, i.e., 
$x_r=(1/\delta_r)y_r$ is a solution to \eqref{lineq}. Since
\eqref{lineq} has a solution, it must hold that $Z\T c=0$. We have
$y_k = H y_k^{(1)} + \delta_k^{(1)} c$ for $k=0,\dots,r$, so that $Z\T
y_k = \delta^{(1)}_kZ^Tc$. As $Z\T c=0$, it follows that $Z\T y_k=0$, $k=0,\dots,r$.

For (b), suppose that $\delta_r = 0$. We have $y_r = H y_r^{(1)} + \delta_r^{(1)} c$, so that
  $Z^Ty_r=\delta^{(1)}_rZ^Tc$. If $\delta_r = 0$, then $0=q_r=Hy_r$ so
  that $y_r=\delta^{(1)}_rZ Z^Tc$. It follows from
  Proposition~\ref{rec} that $y_r\ne 0$ so that $\delta^{(1)}_r \ne0$
  and $Z\T c \ne 0$. A combination of $Hy_r=0$ and $y_r = H y_r^{(1)}
  + \delta_r^{(1)} c$ gives $H( H y_r^{(1)} + \delta_r^{(1)}
  c)=0$. Consequently, since $\delta_r^{(1)}\ne 0$, it holds that
 $H( H x_r^{(1)} + c)=0$ for
 $x_r^{(1)}=(1/\delta_r^{(1)})y_r^{(1)}$. With
 $f(x)=\half\norm{Hx+c}^2_2$, one obtains $\grad f(x)=H(Hx+c)$, so that 
$x_r^{(1)}$ is a global minimizer to $f$ over $\Re^n$. 
\end{proof}

Hence, we have shown that \eqref{lineq} has a solution
if and only if $\delta_r\ne 0$, and that the recursions in Lemma~\ref{rec} can be used to
find a solution if $\delta_r\ne 0$ and a certificate of
incompatibility if $\delta_r=0$.

We can make a few comments on the sequence $\{\delta_k\}$. One can show
that the sequence will never have two zero element in a row.\footnote{This
property is used in composite step biconjugate gradient
method and other look-ahead techniques to show that a composite step or
look-ahead block of size two is sufficient to avoid breakdown, see,
e.g., \cite{BankChan, ChanSzeto}.} Also, if
$\theta_{k-1}$ and $\theta_k$ have the same sign and $\delta_k=0$,
then $\delta_{k+1}\delta_{k-1}<0$. We give direct proofs of these properties, 
using only the recursions of the triples, in Appendix \ref{comment-delta}. 

\subsection{An unnormalized Krylov algorithm}\label{sec-krylovalg}

To summarize the derivation up to this point we now state an algorithm for solving
\eqref{lineq} based on the triples $(q_k, y_k,\delta_k)$, $k=0, \dots,
r$, given by Lemma \ref{rec} using some $\theta_k$ of our
choice. Algorithm \ref{krylovalg} is called a unnormalized Krylov
algorithm\footnote{In the terminology of Gutknecht's survey, \cite{gutknechtacta}, this
  method would be called inconsistent ORes version of the method of
  conjugate gradients.} as it is
the unnormalized vectors $\{q_k\}$, spanning the Krylov subspaces, that drive the
progress of the algorithm.

In the unnormalized setting, the choice of a nonzero $\theta_k$ is in theory arbitrary, but for the
algorithm stated below we have made the choice to let
$\theta_k>0$ such that $\norm{y_{k+1}}_2=\norm{c}_2$.
This choice is well defined since $y_k\ne 0$, $k=1,\dots,r$, by
Lemma~\ref{rec}. 

In theory, triples are generated as long
as $q_k\ne 0$. In the algorithm, we introduce a tolerance such that
the iterations proceed as long as $\norm{q_k}_2> q_{tol}$, where we
let $q_{tol}=\sqrt{\epsilon_M}$, where $\epsilon_M$ is the machine
precision. In theory, we also draw conclusions based on $\delta_r \ne
0$ or $\delta_r = 0$, for this we introduce a tolerance
$\delta_{tol}=\sqrt{\epsilon_M}$.

\begin{algorithm}[htb]
\caption{An unnormalized Krylov algorithm}
\label{krylovalg}
\begin{algorithmic}
\State Input arguments: $H$, $c$;
\State Output arguments: compatible; $x_r$ if compatible=1; $y_r$ if
compatible=0;
\State $q_{tol} \gets$ tolerance on $\norm{q}_2$;
\Comment{Our choice: $q_{tol}=\sqrt{\epsilon_M}$}
\State $\delta_{tol} \gets$ tolerance on $\abs{\delta}$;
\Comment{Our choice: $\delta_{tol}=\sqrt{\epsilon_M}$}
\State $k \gets 0$;
\quad $q_0 \gets c$; \quad $y_0 \gets 0$; \quad $\delta_0 \gets 1$;
\State $\alpha_0 \gets \frac{q_0^THq_0}{q_0^Tq_0}$;
\State $q_{1}\gets \big(-Hq_0+
\alpha_0q_0\big)$; \quad $y_{1} \gets \big(-q_0+
\alpha_0y_0\big)$; 
\quad $\delta_{1} \gets \big(
\alpha_0\delta_0\big)$;

\State $\theta_{0} \gets$ nonzero scalar; 
\Comment{Our choice: $\theta_0=\norm{c}_2/\norm{y_1}_2$}

\State $q_{1}\gets \theta_0 q_1$; \quad 
$y_{1}\gets \theta_0 y_1$; \quad
$\delta_{1}\gets \theta_0 \delta_1$;

\State $k\gets 1$;

\While {$\norm{q_k}_2 > q_{tol}$}
\State $\alpha_k \gets \frac{q_k^TH q_k}{q_k^T q_k}$; \quad
$\beta_{k-1} \gets \frac{q_{k-1}^TH q_k}{q_{k-1}^T q_{k-1}}$;
\State $q_{k+1} \gets \big(-Hq_k+
\alpha_kq_k+
\beta_{k-1}q_{k-1}\big)$; 
    \State $y_{k+1} \gets \big(-q_k+
\alpha_ky_k+
\beta_{k-1}y_{k-1}\big)$; 
\quad $\delta_{k+1}\gets \big(
\alpha_k\delta_k+
\beta_{k-1}\delta_{k-1}\big)$; 
\State $\theta_{k} \gets$ nonzero scalar;
\Comment{Our choice: $\theta_k=\norm{c}_2/\norm{y_{k+1}}_2$}
\State $q_{k+1}\gets \theta_k q_{k+1}$; \quad
$y_{k+1}\gets \theta_k y_{k+1}$; \quad
$\delta_{k+1}\gets \theta_k \delta_{k+1}$;

\State $k \gets k+1$;

\EndWhile

\State $r \gets k$;

\If {$\abs{\delta_r} > \delta_{tol}$} 
\State $x_r\gets \frac{1}{\delta_r} y_r$; \quad compatible $\gets1$;
\Else 
\State compatible $\gets0$;
\EndIf
\end{algorithmic}
\end{algorithm}

By Theorem \ref{conv}, Algorithm \ref{krylovalg} will return either a solution to \eqref{lineq} or a
certificate that the system is incompatible. 

The following small example is chosen to illustrate Algorithm \ref{krylovalg}, with
our choices for $\theta_k >0$, $q_{tol}$ and $\delta_{tol}$, on a
compatible case of \eqref{lineq} where $H$ is a singular matrix. The example
also illustrates the change of sign between $\delta_{k+1}$ and
$\delta_{k-1}$ when $\delta_k=0$.
\begin{example}\label{ex-krylovalg}
Let
\begin{equation*}
c=\left(
\begin{array}{ccccccc} 
 3& 2& 1& 0& -1&-2 &-3
\end{array} \right)^T, \quad
H=diag(c),
\end{equation*}
Algorithm \ref{krylovalg} applied to $H$ and $c$ with $\theta_k>0$
such that $||y_k||=||c||$, $k=1,\dots, r$, and
$q_{tol}=\delta_{tol}=\sqrt{\epsilon_M}$, yields the
following sequences

{\footnotesize   
\begin{verbatim}

q =
      3.0000   -9.0000    2.2678   -2.7046    0.2648   -0.2445    0.0000
      2.0000   -4.0000   -2.2678    5.4912   -1.0591    1.4673         0
      1.0000   -1.0000   -2.2678    2.3768    1.3239   -3.6681         0
           0         0         0         0         0         0         0
     -1.0000   -1.0000    2.2678    2.3768   -1.3239   -3.6681         0
     -2.0000   -4.0000    2.2678    5.4912    1.0591    1.4673         0
     -3.0000   -9.0000   -2.2678   -2.7046   -0.2648   -0.2445   -0.0000

y =
           0   -3.0000    3.4017   -0.9015   -2.2241   -0.0815    2.1602
           0   -2.0000    1.5119    2.7456   -2.8419    0.7336    2.1602
           0   -1.0000    0.3780    2.3768   -0.9885   -3.6681    2.1602
           0         0         0         0         0         0         0
           0    1.0000    0.3780   -2.3768   -0.9885    3.6681    2.1602
           0    2.0000    1.5119   -2.7456   -2.8419   -0.7336    2.1602
           0    3.0000    3.4017    0.9015   -2.2241    0.0815    2.1602

delta =
      1.0000         0   -2.6458         0    2.3123         0   -2.1602
\end{verbatim}
}
\noindent
Hence, $r=6$ and $x_r=(1/\delta_r)y_r=\left(
\begin{array}{ccccccc} 
 -1& -1& -1& 0& -1&-1 &-1
\end{array} \right)^T$.

\end{example}
An example of an incompatible system will be given in Section~\ref{sec-minresalg}.

\subsection{On the case when normalization is well
  defined}\label{sec-cg}

It is well-known that when normalization is well defined and applied
to Algorithm~\ref{krylovalg}, then the method of conjugate gradients,
by Hestenes and Stiefel \cite{HestenesStiefel}, is obtained. In this
case, we denote $(q_k, y_k,\delta_k)$, by $(g_k,x_k,1)$ and the
recursions for $g_k$ and $x_k$ simplify such that it is possible to
obtain a two-term recurrence of a search-direction $p_k$. The
normalization condition for $\theta_k$ is
$\theta_k=1/(\alpha_k+\beta_{k-1})$. We show in
Proposition~\ref{cgstep} that this choice for $\theta_k$ corresponds
exactly to the optimal step-length along $p_k$, which was the
motivation for setting up the recursion \eqref{Q} on that particular
form.

In Lemma~\ref{cor-curv}, we show that if $H \succeq 0$ then $\delta_i
\neq 0$, for $i=0, \dots, r-1$. Also, if $\delta_k > 0$ and
$\delta_{k+1} \neq 0$, then $\delta_{k+1} > 0$ if and only if
$\theta_k >0$. Hence, it holds that for $H \succeq 0$, $\theta_i >0$, $i=0, \dots, r-1$, and $\delta_0>0$,
then $\delta_i>0$, $i=0, \dots, r-1$, and $\delta_r \geq 0$. With the
additional information that $c \in \mathcal{R}(H)$ it holds that $\delta_r >0$ and normalization is possible
in every iteration. On the other hand for $H \succeq 0$, $\theta_i
>0$, $i=0, \dots, r-1$, $\delta_0>0$ and $c \notin
\mathcal{R}(H)$, then $\delta_r =0$ and
normalization is possible at all but the final iteration. Further,
if $H \succ 0$ and $\theta_i >0$, $i=0, \dots, r-1$, then 
$\delta_i > 0$, $i=1, \dots, r$, i.e., $\delta_r \neq 0$ since
\eqref{lineq} with $H \succ 0$ is always compatible.

\section{Connection to the minimum-residual
  method}\label{sec-conminres}

In the case when \eqref{lineq}
is incompatible, instead of just a certificate of
this fact, one would often be interested in a vector $x$ that is
"as good as possible''. The method of choice could then be the minimum
residual method which will return a solution in the compatible case,
and a minimum-residual solution in the incompatible case. 
This method goes back to Lanczos early paper \cite{lanczos52} and
Stiefel \cite{stiefel55}, and the name is adopted from the implementation of
the method, MINRES, by Paige and Saunders, see
\cite{paigesaunders75}. 

For $k=0,\dots,r$, $x_k^{MR}$ is defined as a solution to
$\min_{x \in \mathcal{K}_k(c,H)}\norm{Hx+c}_2^2$, and the corresponding
residual $g_k^{MR}$ is defined as $g_k^{MR}=Hx_k^{MR}+c$. The vectors
$x_k^{MR}$ are uniquely defined for $k=0,\dots,r-1$, and for $k=r$ if
$c\in \mathcal{R}(H)$. For the case $k=r$ and $c\not\in\mathcal{R}(H)$
there is one degree of freedom for $x_r^{MR}$. If $c\in
\mathcal{R}(H)$, then $x_r^{MR}$ solves \eqref{lineq}, and if
$c\not\in \mathcal{R}(H)$, then $x_{r-1}^{MR}$ and $x_r^{MR}$ are both
solutions to $\min_{x\in\Re^n}\norm{Hx+c}_2^2$.

\subsection{Convergence of the minimum-residual method}\label{sec-minres}

In the following theorem, we derive the minimum-residual method based
on the unnormalized Krylov subspace framework. In particular, we give explicit formulas for
$x_k^{MR}$ and $g_k^{MR}$, $k=0,\dots,r$. For the case $k=r$,
$c\not\in\ \mathcal{R}(H)$, we give an explicit formula for $x_r^{MR}$
of minimum Euclidean norm.

\begin{theorem}\label{thm-minres}
  Let $(q_k,y_k,\delta_k)$ be given by Lemma~\ref{rec} for
  $k=0,\dots,r$. Then, for $k=0,\dots,r$, it holds that $x_k^{MR}$ solves
  $\min_{x\in\mathcal{K}_k(c,H)}\norm{Hx+c}_2^2$ if and only if
  $x_k^{MR} =\sum_{i=0}^k \gamma_i y_i$ for some $\gamma_i$,
  $i=0,\dots,k$, that are optimal to 
\begin{equation}\label{minreskrylov}
\begin{array}{ll} 
\minimize{\gamma_0,\dots,\gamma_k} & \disp\half \sum_{i=0}^k
\gamma_i^2 q_i^T q_i \\
\subject & \sum_{i=0}^k \gamma_i \delta_i=1.
\end{array}  
\end{equation}
In particular, $x_k^{MR}$ takes the following form for the mutually
exclusive cases (a) $k<r$; (b) $k=r$ and $\delta_r\ne 0$; and (c)
$k=r$ and $\delta_r=0$.

\def\labelenumi{\alph{enumi})}
\begin{enumerate}

\item
For $k<r$, it holds that
\begin{equation}\label{xmrk}
  x_k^{MR}=\frac{1}{\sum_{j=0}^k\frac{\delta_j^2}{q_j^Tq_j}}\sum_{i=0}^k\frac{\delta_i}{q_i^Tq_i}y_i,
\end{equation}
and $g_k^{MR}=Hx_k^{MR}+c\ne 0$.

\item For $k=r$ and $\delta_r\ne 0$, it holds that $x_r^{MR} =
  ({1}/{\delta_r}){y_r}$ and $g_r^{MR}=Hx_r^{MR}+c=0$, so that
  $x_r^{MR}$ solves \eqref{lineq} and $x_r^{MR}$ is
  identical to $x_r$ of Theorem~\ref{conv}.

\item For $k=r$ and $\delta_r= 0$, it holds that $x_r^{MR} =
  x_{r-1}^{MR} + \gamma_r y_r$, where $\gamma_r$ is an arbitrary
  scalar, and $g_r^{MR}=Hx_r^{MR}+c = g_{r-1}^{MR}\ne 0$. In addition,
  $x_{r-1}^{MR}$ and $x_{r}^{MR}$ solve
  $\min_{x\in\Re^n}\norm{Hx+c}_2^2$. The particular choice
\[
\gamma_r = -\frac{y_r\T x_{r-1}^{MR}}{y_r\T y_r}
\]
makes $x_r^{MR}$ an optimal solution to
$\min_{x\in\Re^n}\norm{Hx+c}_2^2$ of minimum Euclidean norm.
\end{enumerate}
\end{theorem}

\begin{proof}
  Since $q_i$, $i=0,\dots,k$, form an orthogonal basis for
  $\mathcal{K}_{k+1}(c,H)$, an arbitrary vector in
  $\mathcal{K}_{k+1}(c,H)$ can be written as
\begin{equation}\label{eqn-rk}
  g=\sum_{i=0}^k \gamma_i q_i = \sum_{i=0}^k \gamma_i(H
  y_i+\delta_ic) 
= H\big( \sum_{i=0}^k \gamma_i y_i\big)+( \sum_{i=0}^k \gamma_i \delta_i\big)c,
\end{equation}
for some parameters $\gamma_i$, $i=0,\dots,k$. Consequently, the
condition $\sum_{i=0}^k \gamma_i \delta_i=1$ inserted into
(\ref{eqn-rk}) gives $g=H x + c$ with $x = \sum_{i=0}^k \gamma_i y_i$,
i.e., $g$ is an arbitrary vector in $\mathcal{K}_{k+1}(c,H)$ for which
the coefficient in front of $c$ equals one, and $x$ is the
corresponding arbitrary vector in $\mathcal{K}_{k}(c,H)$. Minimizing
the Euclidean norm of such a $g$ gives the quadratic program
\begin{equation}\label{minreskrylovII}
\begin{array}{ll} 
\minimize{g,\gamma_0,\dots,\gamma_k} & \disp\half g\T g \\
\subject         & g = \sum_{i=0}^k \gamma_i q_i, \\
& \sum_{i=0}^k \gamma_i \delta_i=1,
\end{array}  
\end{equation}
so that, by \eqref{eqn-rk}, the optimal values of $\gamma_i$,
$i=0,\dots,k$, give $g_k^{MR}$ as $g_k^{MR} = \sum_{i=0}^k \gamma_i
q_i$ and $x_k^{MR}$ as $x_k^{MR} = \sum_{i=0}^k \gamma_i y_i$.
Elimination of $g$ in \eqref{minreskrylovII}, taking into account the
orthogonality of the $q_i$'s, gives the equivalent problem
\eqref{minreskrylov}.  Also note that since $\delta_0\ne 0$, the
quadratic programs \eqref{minreskrylov} and \eqref{minreskrylovII} are
always feasible, and hence they are well defined.

Let $\mathcal{L}(\gamma, \lambda)$ be the Lagrangian function for
\eqref{minreskrylov}, 
\[
\mathcal{L}(\gamma, \lambda)
  =  \frac{1}{2}\sum_{i=0}^k \gamma_i^2q_i^Tq_i-\lambda\big( \sum_{i=0}^k \gamma_i \delta_i-1\big).
\]
The optimality conditions for \eqref{minreskrylov} are given by
\begin{subequations}\label{optcond}
\begin{eqnarray}
0  &= & \nabla_{\gamma_i} \mathcal{L}(\gamma, \lambda)=
\gamma_iq_i^Tq_i-\lambda \delta_i, \quad i=0,\dots,k, \label{optcond1}\\
0 & = & \nabla_{\lambda} \mathcal{L}(\gamma, \lambda)=1 - \sum_{i=0}^k
\gamma_i \delta_i. \label{optcond2}
\end{eqnarray}
\end{subequations}

\def\labelenumi{\alph{enumi})}

First, for (a), consider the case $k<r$. From \eqref{optcond1} it holds that
\begin{equation}\label{eqn-gammai}
\gamma_i=\lambda \frac{\delta_i}{q_i^Tq_i}, \quad
i= 0,\dots,k,
\end{equation}
which are well defined, since $q_i\ne 0$, $i=0,\dots,r-1$. 
The expression for $\lambda$ is obtained by inserting the
expression for $\gamma_i$, $i=0,\dots,k$, given by (\ref{eqn-gammai})
in (\ref{optcond2}) so that
\begin{equation}\label{eqn-lambda}
1=\sum_{i=0}^k \gamma_i \delta_i=
\sum_{i=0}^k \lambda \frac{\delta_i^2}{q_i^Tq_i}
\text{yielding}
\lambda=\frac{1}{\sum_{i=0}^k\frac{\delta_i^2}{q_i^Tq_i}}.
\end{equation}
Consequently, a combination of (\ref{eqn-gammai}) and
(\ref{eqn-lambda}) gives 
\begin{equation}\label{eqn-gammaiII}
\gamma_i=
\frac{1}{\sum_{j=0}^k\frac{\delta_j^2}{q_j^Tq_j}}
 \frac{\delta_i}{q_i^Tq_i}, \quad
i= 0,\dots,k.
\end{equation}
Hence, by letting $x_k^{MR}=\sum_{i=0}^k \gamma_i y_i$, with
$\gamma_i$ given by (\ref{eqn-gammaiII}), \eqref{xmrk} follows.

Now, for (b), consider the case $k=r$ with $\delta_r\ne 0$. Then,
since $q_r=0$, (\ref{optcond1}) gives $\lambda=0$ and $\gamma_i=0$,
$i=0,\dots,r-1$. Consequently, (\ref{optcond2}) gives
$\gamma_r=1/\delta_r$. Again, by letting $x_r^{MR}=\sum_{i=0}^r
\gamma_i y_i$, it holds that $x_k^{MR}=(1/\delta_r)y_r$, for which
$g_k^{MR}=Hx_k^{MR}+c=0$, so that the optimal value is zero in
(\ref{minreskrylov}) and $x_r^{MR}$ solves (\ref{lineq}).

Finally, for (c), consider the case $k=r$ with
$\delta_r=0$. Theorem~\ref{conv} shows that there exists an
$x_r^{(1)}\in\mathcal{K}_{r-1}(c, H)$ that solves
$\min_{x\in\Re^n}\norm{Hx+c}_2^2$. Consequently, since
$x_r^{(1)}\in{K}_{r-1}(c, H)$, it follows from (a) that it must hold that
$x_r^{(1)}=x_{r-1}^{MR}$ so that $x_{r-1}^{MR}$ solves
$\min_{x\in\Re^n}\norm{Hx+c}_2^2$.

  For $k=r$, the optimality conditions (\ref{optcond}) are equivalent
  to when $k=r-1$, just with the additional information that
  $\gamma_r$ is arbitrary. Hence, $x_r^{MR}=x_{r-1}^{MR} +
  \gamma_r y_r$ and $g_r^{MR}=Hx_r^{MR}+c=g_{r-1}^{MR}$ for arbitrary $\gamma_r$
  since $H y_r=0$. 

Regardless of the value of $\gamma_r$, the range-space component of
$x_r^{MR}$ is unique, since Theorem~\ref{conv} gives $H y_r=0$.
  For the remainder of the proof, let $x_r^{MR}=x_{r-1}^{MR} +
  \gamma_r y_r$ for the particular choice $\gamma_r =-(y_r^T
  x_{r-1}^{MR})/(y_r^Ty_r)$, so that $y_r^Tx_r^{MR}=0$. We will show
  that for this particular choice the null-space component of $x_r^{MR}$ is
  zero, and hence $x_r^{MR}$ is an optimal solution to
  $\min_{x\in\Re^n}\norm{Hx+c}_2^2$ of minimum Euclidean norm. 

Since
  $y_k=Hy_k^{(1)}+\delta_k^{(1)} c$, it follows that $Z^T y_k =
  \delta_k^{(1)} Z^T c$, $k=0,\dots,r$. Consequently, since
  $x_{r}^{MR}$ is formed as a linear combination of $y_k$,
  $k=0,\dots,r$, it holds that $Z\T x_{r}^{MR}$ is parallel to $Z\T
  c$. But $0=y_r^Tx_r^{MR}=(\delta_r^{(1)}Z^Tc)^T Z^Tx^{MR}=0$, so that $Z^Tx^{MR}$ is
  also orthogonal to $Z\T c$. By Theorem~\ref{conv}, $\delta_r^{(1)}
  \neq 0$ and $Z\T c\ne 0$, so
  that $Z^Tx^{MR}=0$. Hence, the null space component of $x_r^{MR}$ is
  zero. 
\end{proof}

Note that at an iteration $k$ at which $\delta_k=0$ and $q_k\ne0$,
it holds that $x_k^{MR}=x_{k-1}^{MR}$ so that the iterate is
unchanged. This is referred to as stagnation, and in accordance with Brown
\cite{brown} it holds that the unnormalized Krylov method and the
minimum-residual method form a pair, see, e.g., \cite[Proposition
6.17]{saad}. In the framework of this paper, it holds that normalization is not possible
at step $k$ in the Krylov method if and only if there is stagnation
in the minimum-residual method. Note that this cannot happen at two consecutive
iterations. If $q_k\ne0$, all information from the problem has not
been extracted even if $\delta_k=0$. Only in the case when $k=r$,
global information is obtained, and it is determined whether
\eqref{lineq} has a solution or not.

In the following corollary of Theorem~\ref{thm-minres} we state explicit recursions for the minimum-residual method.

\begin{corollary}\label{rec-minres}
Let $(q_k,y_k,\delta_k)$ be given by Lemma~\ref{rec} and let $x_k^{MR}$ be given by Theorem~\ref{thm-minres} for
$k=0,\dots,r$. If $\delta_0^{MR}=\delta_0^2$, $y_0^{MR}=\delta_0 y_0$,
\[
\delta_k^{MR}=q_k^Tq_k \sum_{i=0}^{k-1}\frac{\delta_i^2}{q_i^Tq_i} + \delta_k^2
\text{and} y_k^{MR}=q_k^Tq_k
\sum_{i=0}^{k-1}\frac{\delta_i}{q_i^Tq_i}y_i + \delta_k y_k,
\quad k=1,\dots,r,
\]
then
\[
x_{k}^{MR}=\frac{1}{\delta_{k}^{MR}}y_{k}^{MR}, \quad k=0,\dots,r-1
\text{and} k=r \mbox{\ if\ } \delta_r\ne 0.
\]
In addition, it holds that
\[
\delta_{k+1}^{MR}=\frac{q_{k+1}^Tq_{k+1}}{q_k^Tq_k}\delta_k^{MR}+\delta_{k+1}^2
\text{and}
y_{k+1}^{MR}=\frac{q_{k+1}^Tq_{k+1}}{q_k^Tq_k}y_k^{MR}+\delta_{k+1}y_{k+1},
\]
for $k=0,\dots,r-1$.
\end{corollary}

\begin{proof}
For $k=0,\dots,r-1$, the expressions for $\delta_k^{MR}$ and
$y_k^{MR}$ give
\[
\frac{1}{q_k^T q_k}\delta_k^{MR}=\sum_{i=0}^{k}\frac{\delta_i^2}{q_i^Tq_i}
\text{and} \frac{1}{q_k^T q_k}y_k^{MR}=\sum_{i=0}^{k}\frac{\delta_i}{q_i^Tq_i}y_i.
\]
Note that $\delta_0\ne 0$ and $q_i\ne 0$, $i=0,\dots,k$ implies
$\delta_k^{MR}>0$ for $k<r$. Hence, Theorem~\ref{thm-minres} gives
$x_k^{MR}=(1/\delta_k^{MR})y_k^{MR}$. If $k=r$ and $\delta_r \neq 0$,
the expressions for $\delta_r^{MR}$ and $y_r^{MR}$ give
\[
\delta_r^{MR}=\delta_r^2 > 0 \text{and} y_r^{MR}=\delta_r y_r,
\]
so that Theorem~\ref{thm-minres} gives
$x_r^{MR}=(1/\delta_r^{MR})y_r^{MR}$ also for this case.

The recursions for $y_{k+1}^{MR}$ and $\delta_{k+1}^{MR}$,
$k=0,\dots,r-1$, are straightforward to obtain.
\end{proof}

The recursion for $x_r^{MR}$, based on $x_{r-1}^{MR}$ and
$(q_r,y_r,\delta_r)$, for the case $\delta_r=0$ is given in
Theorem~\ref{thm-minres} and it is not reiterated in
Corollary~\ref{rec-minres}.

Note that the expressions in Theorem~\ref{thm-minres} and
Corollary~\ref{rec-minres} for $x_k^{MR}$, $k=0,\dots, r$, are
independent of the scaling of $(q_k,y_k,\delta_k)$. Hence, if
$H\succeq 0$ and $c\in\mathcal{R}(H)$ then normalization
is well defined so that $(g_k, x_k,1)$ may be used to give $x_k^{MR}$, $k=0, \dots, r-1$, as 
convex combinations of $x_i$, $i=0,\dots, k$, respectively.

\subsection{A minimum-residual algorithm based on the unnormalized Krylov method}\label{sec-minresalg} 

To summarize we next state an algorithm for the minimum-residual method based
on Algorithm~\ref{krylovalg} and extended with the recursions in 
Corollary~\ref{rec-minres}. 

\begin{algorithm}[!htb]
\caption{A minimum-residual algorithm based on the unnormalized Krylov method}
\label{minresalg}
\begin{algorithmic}
\State Input arguments: $H$, $c$;
\State Output arguments: $x_r^{MR}$, $g_r^{MR}$, compatible; \quad
($x_r$ or $y_r$ if compatible=1 or 0;)
\State Run Algorithm \ref{krylovalg} with the extra initialization
\State $\quad y_0^{MR} \gets \delta_0 y_0$; \quad $\delta_0^{MR} \gets
\delta_0^2$; 
\quad $x_0^{MR} \gets \frac{1}{\delta_{0}^{MR}}y_{0}^{MR}$;
\quad $g_0^{MR} \gets H x_{0}^{MR} +c$;
\State For $k=1$ calculate in addition
\State $\quad y_1^{MR} \gets \frac{q_1^Tq_1}{q_0^Tq_0}y_0^{MR} + \delta_1
y_1$; \quad $\delta_1^{MR} \gets \frac{q_1^Tq_1}{q_0^Tq_0}\delta_0^{MR} +
\delta_1^2$; 
\State $\quad x_1^{MR} \gets \frac{1}{\delta_{1}^{MR}}y_{1}^{MR}$;
\quad $g_1^{MR} \gets  Hx_{1}^{MR}+c$;
\State For $k>1$ until termination calculate in addition
\State $\quad y_{k+1}^{MR} \gets 
\frac{q_{k+1}^Tq_{k+1}}{q_{k}^Tq_{k}}y_{k}^{MR} + \delta_{k+1} y_{k+1}$; \quad
$\delta_{k+1}^{MR} \gets \frac{q_{k+1}^Tq_{k+1}}{q_{k}^Tq_{k}}\delta_{k}^{MR} +
\delta_{k+1}^2$; 
\State $\quad x_{k+1}^{MR}\gets\frac{1}{\delta_{k+1}^{MR}}y_{k+1}^{MR}$;
\quad $g_{k+1}^{MR}\gets Hx_{k+1}^{MR}+c$;
\State At termination, if $\abs{\delta_r} < \delta_{tol}$, calculate in addition
\State $\quad x_{r}^{MR}\gets x_{r-1}^{MR}- \frac{y_r\T x_{r-1}^{MR}}{y_r\T y_r} y_r$; \quad compatible $\gets0$;
\end{algorithmic}
\end{algorithm}

Hence, for a compatible system \eqref{lineq} Algorithm~\ref{minresalg}
gives the same solution $x_r$ as Algorithm~\ref{krylovalg}, and in
addition it calculates $x_r^{MR}$. They are both estimates of a
solution to \eqref{lineq}. Further, if \eqref{lineq} is incompatible
then Algorithm~\ref{minresalg} delivers $x_r^{MR}$, an optimal
solution to $\min_{x\in\Re^n}\norm{Hx+c}_2^2$ of minimum Euclidean
norm, in addition to the certificate of incompatibility.

Next we observe another small example chosen to illustrates Algorithm \ref{minresalg} with
our choices for $\theta_k >0$, $q_{tol}$ and $\delta_{tol}$, on a case
when \eqref{lineq} is incompatible, i.e. $c \notin \mathcal{R}(H)$.

\begin{example}
Let
\begin{equation*}
c=\left(
\begin{array}{ccccccc} 
 3& 2& 1& 1& -1&-2 &-3
\end{array} \right)^T, \quad
H=diag \left(
\begin{array}{ccccccc} 
 5& 2& 1& 0& -1&-2 &-3
\end{array} \right),
\end{equation*}
Algorithm \ref{minresalg} applied to $H$ and $c$ with $\theta_k>0$
such that $||y_k||=||c||$, $k=1,\dots, r$, and
$q_{tol}=\delta_{tol}=\sqrt{\epsilon_M}$, yields the
following sequences
{\footnotesize   
\begin{verbatim}
q =

 3.0000  -13.1379    3.5628   -0.8597    0.1063   -0.0181    0.0017   -0.0000
 2.0000   -2.7586   -5.7676    3.9832   -1.3787    0.6372   -0.1470   -0.0000
 1.0000   -0.3793   -3.1464    0.1039    1.8638   -2.5737    1.1021    0.0000
 1.0000    0.6207   -2.8617   -1.7605    2.2573    0.5896   -1.7634   -0.0000
-1.0000   -1.6207    2.0296    2.7934   -0.6882   -2.4489   -1.4695    0.0000
-2.0000   -5.2414    1.3007    4.3735    2.1842    1.1548    0.3149   -0.0000
-3.0000  -10.8621   -3.8286   -2.6032   -0.6658   -0.2082   -0.0367    0.0000

y =

      0   -3.0000    2.4296    0.8844   -1.3331   -0.3574    1.0584   -0.0000
      0   -2.0000   -0.0222    3.7521   -2.9466   -0.2710    1.6899         0
      0   -1.0000   -0.2847    1.8644   -0.3935   -3.1633    2.8655    0.0000
      0   -1.0000   -0.5584    1.5833    0.7502   -3.7018   -0.7931    5.3852
      0    1.0000    0.8320   -1.0329   -1.5691    1.8593    3.2329    0.0000
      0    2.0000    2.2113   -0.4262   -3.3494   -1.1670    1.6060    0.0000
      0    3.0000    4.1379    2.6283   -2.0353   -0.5202    1.7756    0.0000

delta =

 1.0000    0.6207   -2.8617   -1.7605    2.2573    0.5896   -1.7634   -0.0000

xMR =
 
      0   -0.1588   -0.6633   -0.6143   -0.5995   -0.5998   -0.6000   -0.6000
      0   -0.1059   -0.0228   -0.6647   -1.0640   -1.0371   -1.0000   -1.0000
      0   -0.0529    0.0585   -0.2817   -0.2148   -0.4441   -1.0000   -1.0000
      0   -0.0529    0.1284   -0.1845    0.1376   -0.1481    0.1333   -0.0000
      0    0.0529   -0.1983    0.0407   -0.4178   -0.2588   -1.0000   -1.0000
      0    0.1059   -0.5364   -0.2994   -1.0375   -1.0794   -1.0000   -1.0000
      0    0.1588   -1.0143   -1.1600   -0.9990   -0.9938   -1.0000   -1.0000
\end{verbatim}
}

\noindent 
Hence, $r=7$ and $x_r^{MR}=\left(
\begin{array}{ccccccc} 
 -0.6& -1& -1& 0& -1&-1 &-1
\end{array} \right)^T$. Note that, since $\abs{\delta_r}< \delta_{tol}$, the system is
considered incompatible and $x_r^{MR}$ is the optimal solution to
$\min_{x\in\Re^n}\norm{Hx+c}_2^2$ of minimum Euclidean norm
and $\norm{Hx_r^{MR}+c}_2^2=1$.
\end{example}

\section{Summary and conclusion}

By making use of an unnormalized Krylov subspace framework for solving symmetric system of linear
equations, as proposed by Gutknecht \cite{Gutknecht90,
  Gutknecht92}, we show how to determine, in exact arithmetic, if
the system is compatible or incompatible. In the compatible case, a
solution is given. In the incompatible case, a
certificate of incompatibility is obtained. The basis of this framework
are the triples $(q_k,y_k,\delta_k)$, $k=0,\dots,r$, given by
Lemma~\ref{rec}, that are uniquely defined up to a scaling. 
Our results include and handle the case of a singular matrix $H$. To
the best of our knowledge this is not covered in any previous work.

We have also put the minimum-residual method in this framework and
provided explicit formulas for the iterations. Again, the
analysis is based on the triples. In the case of an incompatible system, our
analysis gives an expression for $x_r^{MR}$ of minimum Euclidean norm. The original
implementation of MINRES by Paige and Saunders, \cite{paigesaunders75}, did not deliver the
optimal solution to $\min_{x\in\Re^n}\norm{Hx+c}_2^2$ of minimum
Euclidean norm. In \cite{choipaigesaunders}, Choi, Paige and Saunders present a MINRES-like
algorithm, MINRES-QLP, that does.

One may observe that an alternative to using the minimum-residual
iterations would be to consider recursions for $y^{(1)}_k$ and
$\delta^{(1)}_k$ as given by \eqref{eqn-y1} and then calculate
$x_{r-1}^{MR}=(1/\delta^{(1)}_r) y^{(1)}_r$, according to the analysis
in the proof of Theorems~\ref{conv} and \ref{thm-minres}. However,
such an approach would not automatically yield the estimates
$x_k^{MR}$, $k=0, \dots, r-2$.

One could also note that the method of conjugate gradients may be viewed
as trying to solve the minimum-residual problem \eqref{minreskrylov}
in the situation where only the present triple $(q_k,y_k,\delta_k)$ is
allowed in the linear combination, i.e., $\gamma_i=0$,
$i=0,\dots,k-1$. This problem is then not necessarily feasible. It will
be infeasible exactly when $\delta_k=0$. One could think of methods
other than the minimum-residual method which use a linear combination
of more than one triple. It would
suffice to use two consecutive triples, since it cannot hold that
$\delta_{k-1}=0$ and $\delta_k=0$. 

Finally, we want to stress that this paper is meant to give insight
into the unnormalized Krylov subspace framework, in exact
arithmetic. In finite precision, the
unnormalized Krylov method would inherit deficiencies of any method based
on a Lanczos process such as loss of orthogonality of the generated
vectors. It is beyond the scope of the present paper to make such an
analysis, see, e.g., \cite{hanke, meurant, paige76, reid71}. The
theory of our paper is based on determining if certain quantities are
zero or not. In our algorithms, we have made choices on
optimality tolerances that are not meant to be universal. To obtain a
fully functioning algorithm, the issue of determining if a quantity is
near-zero would need to be considered more in detail. Also, we have
based our analysis on the triples, so that termination of
Algorithm~\ref{minresalg} is based on $q_k$ and $\delta_k$. In
practice, one should probably also consider $g_k^{MR}$.

Further, the use of pre-conditioning is not explored in this paper,
for this subject see, e.g., \cite{evans68, evans73, axelsson74}.

\appendix
\section{Appendix}
\subsection{A result on the Lanczos vectors}\label{comment-q}

For completeness, we review a result that characterizes the properties
of $q_k$ expressed as in \eqref{eqn-qk}, needed for the analysis.

\begin{lemma}\label{lem-qkappendix}
  Let $r$ denote the smallest positive integer $k$ for which $\mathcal{K}_{k+1}(c,H)= \mathcal{K}_k(c, H)$.
For an index $k$ such that $1\le k\le r$, let $q_k \in \mathcal{K}_{k+1}(c,
H)\cap \mathcal{K}_k(c, H)^{\perp}$, be expressed as in
\eqref{eqn-qk}. Then, the scalars $\delta_k^{(j)}$, $j=0,\dots,k$, are
uniquely determined up to a nonzero scaling. In addition, if
$\delta_k^{(k)}\ne 0$ it holds that $q_k=0$ if and only if
$\mathcal{K}_{k+1}(c,H)= \mathcal{K}_k(c, H)$, i.e.,
if $k=r$.
\end{lemma}

\begin{proof}
  Assume that $q_k \in \mathcal{K}_{k+1}(c,H)\cap \mathcal{K}_k(c,
  H)^{\perp}$ is expressed as in \eqref{eqn-qk}. If $k<r$, then $c$,
  $Hc$, $H^2c$, \dots, $H^{k}c$ are linearly independent. Hence,
  $\delta_k^{(j)}$, $j=0,\dots,k$, are uniquely determined by
  $q_k$. Consequently, as $q_k$ is uniquely defined up to a nonzero
  scaling, then so are $\delta_k^{(j)}$, $j=0,\dots,k$. For $k=r$, we
  have $q_r=0$ so that
\begin{equation}\label{eqn-r}
-\delta_r^{(r)} H^r c =\sum_{j=0}^{r-1}\delta^{(j)}_rH^jc.
\end{equation}
By the definition of $r$, it holds that $c$, $Hc$, $H^2c$, \dots,
$H^{r-1}c$ are linearly independent. Hence, \eqref{eqn-r} shows that a
fixed $\delta_r^{(r)}$ uniquely determines $\delta^{(j)}_r$,
$j=1,\dots,r-1$. Consequently, a scaling of $\delta^{(r)}_r$ gives a
corresponding scaling of $\delta_r^{(j)}$,
$j=0,\dots,r-1$. Thus, $\delta_r^{(j)}$, $j=0,\dots,r$, are uniquely
determined up to a common scaling. 

Finally, assume that $\delta_k^{(k)}\ne 0$. By definition
$\mathcal{K}_{k+1}(c,H)=\mathcal{K}_k(c,H)$ implies
$q_k=0$. To show the converse, assume that $q_k=0$. Then,
\begin{equation}\label{eqn-k}
-\delta_k^{(k)} H^k c =\sum_{j=0}^{k-1}\delta^{(j)}_kH^jc.
\end{equation}
If $\delta_k^{(k)}\ne 0$, then \eqref{eqn-k} implies $H^k c\in
span\{c, Hc, H^2c, \dots, H^{k-1}c\}$, i.e.,
$\mathcal{K}_{k+1}(c,H)=\mathcal{K}_k(c, H)$, 
completing the proof.
\end{proof}

\subsection{Properties of the sequence $\{\delta_k\}$}\label{comment-delta}

In the following proposition it is shown that
the sequence $\{\delta_k\}$ can not have two zero elements in a row.

\begin{proposition}\label{deltasign}
  Let $(q_k, y_k,\delta_k)$, $k=0, \dots, r$, be given by
  Lemma \ref{rec}. If $q_k \neq 0$ and $\delta_k=0$, then
\[
\delta_{k+1}
= -\frac{\theta_k}{\theta_{k-1}} \frac{q_{k}^Tq_{k}}{q_{k-1}^T
  q_{k-1}}\delta_{k-1} \ne 0.
\]
\end{proposition}

\begin{proof}
 By Proposition \ref{genq} it holds that
\[
q_{k}^T q_{k} = -\theta_{k-1}q_{k}^TH q_{k-1}, 
\]
and, taking into account $\delta_k=0$, the expression for
$\delta_{k+1}$ from Proposition~\ref{rec} gives
\begin{equation}\label{eqn-deltaconn}
\delta_{k+1}=\theta_k \frac{q_{k-1}^TH q_k}{q_{k-1}^T
  q_{k-1}}\delta_{k-1}
= -\frac{\theta_k}{\theta_{k-1}} \frac{q_{k}^Tq_{k}}{q_{k-1}^T q_{k-1}}\delta_{k-1},
\end{equation}
giving the required expression for $\delta_{k+1}$.

It remains to show that $\delta_{k+1}\ne 0$. First, assume that $k=1$
so that $\delta_1=0$. Then, since $\theta_{1}\ne 0$, $\theta_{0}\ne 0$
and $\delta_0=1$, \eqref{eqn-deltaconn} gives $\delta_{2}\ne 0$. Now
assume that $k>1$. Assume, to get a contradiction, that
$\delta_{k+1}=0$. Then, since $\theta_{k}\ne 0$, $\theta_{k-1}\ne 0$,
\eqref{eqn-deltaconn} gives $\delta_{k-1}= 0$. We may then repeat the
same argument to obtain $\delta_{i}=0$, $i=1,\dots,k$. But this gives
a contradiction, as $\delta_1=0$ implies
$\delta_2\ne 0$. Hence, it must hold that $\delta_{k+1}\ne 0$, as
required.
\end{proof}

Based on Proposition~\ref{deltasign} the following corollary states
that if $\theta_{k-1}$ and $\theta_k$ have the same sign and
$\delta_k=0$, then $\delta_{k+1}$ and $\delta_{k-1}$ will have
opposite signs.

\begin{corollary}\label{deltasignII}
  Let $(q_k, y_k,\delta_k)$, $k=0, \dots, r$, be given by Proposition
  \ref{rec} with $\theta_{k-1}$ and $\theta_k$ of the same sign. If
  $q_k \neq 0$ and $\delta_k=0$, then $\delta_{k+1}\delta_{k-1}<0$.
\end{corollary}

The following lemma states an expression for the triples that is used
 in showing properties of the signs of $\delta_k$ and $\theta_k$
for the case when $H$ is positive semidefinite.

\begin{lemma}\label{lem-curv}
  Let $(q_k, y_k,\delta_k)$, $k=0, \dots, r$, be given by
  Lemma \ref{rec}. If $\delta_k \ne 0$ and $k < r$, then
\begin{equation}\label{exp}
\big(
y_{k+1}-\frac{\delta_{k+1}}{\delta_{k}}y_k\big)^TH \big(
y_{k+1}-\frac{\delta_{k+1}}{\delta_{k}}y_k\big) = 
\theta_k \frac{\delta_{k+1}}{\delta_{k}}q_{k}^Tq_{k}.
\end{equation}
\end{lemma}
\begin{proof}
Eliminating $c$ from the difference of $q_{k+1}$ and $q_k$ yields
\begin{equation}\label{limdiff}
q_{k+1}-\frac{\delta_{k+1}}{\delta_{k}}q_k=H\big( y_{k+1}-\frac{\delta_{k+1}}{\delta_{k}}y_k\big).
\end{equation}
Then pre-multiplication of \eqref{limdiff} with $(
y_{k+1}-\frac{\delta_{k+1}}{\delta_{k}}y_k)^T$ yields
\begin{equation}\label{conn1}
\big(
y_{k+1}-\frac{\delta_{k+1}}{\delta_{k}}y_k\big)^T\big(q_{k+1}-\frac{\delta_{k+1}}{\delta_{k}}q_k\big) =
\big(
y_{k+1}-\frac{\delta_{k+1}}{\delta_{k}}y_k\big)^TH \big(
y_{k+1}-\frac{\delta_{k+1}}{\delta_{k}}y_k\big).
\end{equation}
Since $q_{k+1}$ is orthogonal to $y_{k+1}$ and
$y_{k}$, and since $q_{k}$ is orthogonal to $y_{k}$, \eqref{conn1} becomes
\begin{equation}\label{conn2}
-\frac{\delta_{k+1}}{\delta_{k}}y_{k+1}^T q_{k} =\big(
y_{k+1}-\frac{\delta_{k+1}}{\delta_{k}}y_k\big)^TH \big(
y_{k+1}-\frac{\delta_{k+1}}{\delta_{k}}y_k\big).
\end{equation}
Hence, by Proposition \ref{genq} and since $q_{k}$ is orthogonal to
$y_{k}$ and $y_{k-1}$, \eqref{conn2} may be written as
\[
\theta_k\frac{\delta_{k+1}}{\delta_{k}}q_{k}^T
q_{k}=\big(
y_{k+1}-\frac{\delta_{k+1}}{\delta_{k}}y_k\big)^TH \big(
y_{k+1}-\frac{\delta_{k+1}}{\delta_{k}}y_k\big),
\]
hence \eqref{exp} is obtained.
\end{proof}

The following lemma gives some results on the behavior of the sequence
of $\{\delta_k\}$ in connection to the sign of $\theta_k$ for the case when $H \succeq 0$.

\begin{lemma}\label{cor-curv}
  Let $(q_k, y_k,\delta_k)$, $k=0, \dots, r$, be given by
  Lemma \ref{rec}. Assume that $H \succeq 0$. Then $\delta_k \ne
  0$ for $k < r$. If $\delta_k > 0$ and $\delta_{k+1}\ne 0$, then
  $\delta_{k+1} > 0$ if and only if $\theta_k>0$.
\end{lemma}

\begin{proof}
Assume that $\delta_k=0$ for $k < r$, then $q_{k}= Hy_k$, hence
pre-multiplication with $y_k^T$ yields
$0=y_k^T q_{k}=y_k^T Hy_k, 
$
since $q_k$ is orthogonal to $y_k$. Then, since $H \succeq 0$, it follows that $Hy_k=0$ and hence
$q_{k}=0$. Since $q_k \neq 0$ for $k < r$, the assumption yields a contradiction. Hence,
$\delta_k\neq 0$, $k <r$.

Next suppose that $\delta_k > 0$ and $\delta_{k+1}\ne 0$. Since $H\succeq 0$, Lemma~\ref{lem-curv} gives
\begin{equation}\label{ineq}
\theta_k\frac{\delta_{k+1}}{\delta_{k}}q_{k}^T
q_{k} \geq 0,
\end{equation}
which implies that $\delta_{k+1}$ and $\delta_{k}$ 
have the same sign if and only if $\theta_k >0$. Hence, if
$\delta_k> 0$, then $\delta_{k+1}>0$ if and only if $\theta_k >0$. 
\end{proof}

The relation of the signs in Lemma~\ref{cor-curv}
is a consequence of our choice of the minus-sign in \eqref{q}. Otherwise
$\delta_k$ would alternate sign in each iteration for $\theta_k >0$ and $H
\succeq 0$.

A consequence of Lemma~\ref{cor-curv} is that if $\theta_k$ is chosen positive for
$k=0,\dots,r$, then $\delta_k\le 0$ for some $k$ implies $H\not\succ
0$ and $\delta_k<0$ for some $k$ implies $H\not\succeq 0$.

\subsection{The method of conjugate gradients}\label{app-cg}
If normalization is well defined and applied to Algorithm
\ref{krylovalg}, then one obtains the method of conjugate gradients, by Hestenes and
Stiefel \cite{HestenesStiefel}. For an introduction to the method of conjugate gradients see, e.g.,
\cite{luenberger,cgwopain}. This method is usually defined for the case where
$H\succ 0$. In the method of conjugate gradients, an iterate $x_k$ is defined as the solution to
$\min_{\mathcal{K}_k(c,H)}\half x^T H x+c^T x$, and $g_k=H x_k +c$ for
$k=0,\dots,r$, i.e., $g_k\in \mathcal{K}_{k+1}(c,H)\cap
\mathcal{K}_k(c, H)^{\perp}$.

In the setting of this paper, it is equivalent to generating triples $(q_k, y_k,
\delta_k)$, $k=0,\dots,r$, given by Lemma \ref{rec}, selecting
the scaling $\theta_k$ in \eqref{normcond} such that $\delta_k=1$, for
all $k$. With the additional assumption
$H \succeq 0$, Lemma \ref{cor-curv} gives $\delta_k \ne 0$, $k=0,
\dots, r-1$. If $c \in \mathcal{R}(H)$, i.e., \eqref{lineq} is
compatible, then Theorem~\ref{conv} ensures that also
$\delta_r\ne0$. Further, if  $H
\succeq 0$ and $c \not\in \mathcal{R}(H)$, normalization will be
well defined in all except the very last iteration.

For completeness, in the following proposition we show that when the
normalization condition is satisfied, $\theta_k$ is exactly the
step-length along the search-direction $p_k$ in iteration $k$, so that
the usual line-search description of the method of conjugate
gradients, see, e.g., \cite{saad}, follows.

\begin{proposition}\label{cgstep}
Assume that $H \succeq 0$ and $c \in \mathcal{R}(H)$. If
  $(q_k,y_k,\delta_k)$, $k=0,\dots,r$, are given by Lemma~\ref{rec},
  for the choice of $\theta_k$ in \eqref{normcond}, then $\delta_k=1$,
  $k=1, \dots, r$, and $\theta_k>0$, $k=0, \dots, r-1$. Hence, denoting $(q_k, y_k,
\delta_k)$ by $(g_k, x_k, 1)$, for 
\begin{eqnarray*}
  p_0 & = & -g_0, \quad p_k=-g_k+\frac{g_k^Tg_k}{g_{k-1}^Tg_{k-1}} p_{k-1}, \quad
  k=1,\dots,r-1. 
\end{eqnarray*}
it holds that
\begin{equation*}
x_{k+1} = x_k + \theta_k p_k, \quad k=0,\dots,r-1, 
\end{equation*}
\begin{equation*}
g_{k+1} = g_k + \theta_k H p_k, \quad k=0,\dots,r-1,
\end{equation*}
and further,
\begin{eqnarray*}
  \theta_k & = & -\frac{g_k^T p_k}{p_k^THp_k}, \quad k=0,\dots,r-1.
\end{eqnarray*}
\end{proposition}

\begin{proof}
Let $(q_0, y_0, \delta_0)=(c, 0, 1)$, then with $\theta_k$ as in
\eqref{normcond}, i.e.,
\begin{equation*}
\theta_0  =  \frac{1}{\alpha_0}, \quad \theta_k    = 
\frac{1}{\alpha_k+\beta_{k-1}},
\quad k=1,\dots,r-1,
\end{equation*}
where $\alpha_k$, $k=0, \dots, r-1$, and $\beta_{k-1}$, $k=1, \dots,
r-1$, are given by \eqref{alphabeta}, the recursions of
Lemma~\ref{rec} yield $\delta_k=1$, $k=0, \dots, r-1$. Hence, by
Lemma~\ref{cor-curv}, $\theta_k>0$, $k=0, \dots, r-1$. Denoting $(q_k,
y_k, \delta_k)$ by $(g_k, x_k, 1)$, the recursions for $x_k$ and $g_k$
of Lemma~\ref{rec} are then given by
\begin{eqnarray*}
x_{1}  &=&  x_0 + \theta_0 (-g_0), \\
x_{k+1} & = & x_k + \theta_k \big(-g_k -
\beta_{k-1}(x_k-x_{k-1})\big), \quad
k=1,\dots,r-1, 
\end{eqnarray*}
and
\begin{eqnarray*}
g_1  &  = &Hx_1+c = g_0+\theta_0(-Hg_0), \\
g_{k+1} & =  &Hx_{k+1}+c = g_k+\theta_k\big(
-Hg_k-\beta_{k-1}(g_k-g_{k-1})\big), \  k=1,\dots,r-1.
\end{eqnarray*}
For $p_k = (1/\theta_k) (x_{k+1}-x_{k})$,
  $k=0,\dots,r-1$, the above recursions give
\[
p_0=-g_0, \text{and} p_k=-g_k-
\beta_{k-1}(x_k-x_{k-1}), \quad k=1,
\dots, r-1.
\]
By Proposition \ref{genq} it holds that $g_k^Tg_k=-\theta_{k-1} g_{k}^THg_{k-1}$,
and therefore,
$$
\beta_{k-1}=\frac{g_{k-1}^TH g_k}{g_{k-1}^T g_{k-1}}=-\frac{1}{\theta_{k-1}}\frac{g_{k}^Tg_k}{g_{k-1}^T
  g_{k-1}},
$$ 
hence
\[
p_k=-g_k-
\beta_{k-1}\theta_{k-1} p_{k-1}=-g_k+
\frac{g_{k}^Tg_k}{g_{k-1}^T
  g_{k-1}}p_{k-1}, \quad k=1,
\dots, r-1.
\]
Consequently, since $g_{k+1}-g_k=H(x_{k+1}-x_k)=\theta_k H p_k$,
$k=0,\dots,r-1$, it holds that $g_{k+1} = g_k + \theta_k H p_k$,
$k=0,\dots,r-1$. Further, $g_{k+1}^T p_k=0$ since $g_{k+1}\in
\mathcal{K}_{k+1}(c,H)\cap \mathcal{K}_k(c, H)^{\perp}$ and
$p_k\in\mathcal{K}_k(c,H)$.
Hence, $0=g_{k+1}^T p_k=g_k^Tp_k + \theta_k p_k^TH p_k$
yields
$$
\theta_k=-\frac{g_k^Tp_k}{p_k^THp_k}, \quad k=0,\dots,r-1,
$$
completing the proof.
\end{proof}

\section*{Acknowledgement}

This research was partially supported by the Swedish Research Council (VR).

\bibliography{references}{}

\begin{thebibliography}{10}

\bibitem{axelsson74}
{\sc O.~Axelsson}, {\em On the efficiency of a class of {$A$}-stable methods},
  Nordisk Tidskr. Informationsbehandling (BIT), 14 (1974), pp.~279--287.

\bibitem{BankChan}
{\sc R.~E. Bank and T.~F. Chan}, {\em An analysis of the composite step
  biconjugate gradient method}, Numer. Math., 66 (1993), pp.~295--319.

\bibitem{brown}
{\sc P.~N. Brown}, {\em A theoretical comparison of the {A}rnoldi and {GMRES}
  algorithms}, SIAM J. Sci. Statist. Comput., 12 (1991), pp.~58--78.

\bibitem{ChanSzeto}
{\sc T.~F. Chan and T.~Szeto}, {\em Composite step product methods for solving
  nonsymmetric linear systems}, SIAM J. Sci. Comput., 17 (1996),
  pp.~1491--1508.

\bibitem{choipaigesaunders}
{\sc S.-C.~T. Choi, C.~C. Paige, and M.~A. Saunders}, {\em M{INRES}-{QLP}: a
  {K}rylov subspace method for indefinite or singular symmetric systems}, SIAM
  J. Sci. Comput., 33 (2011), pp.~1810--1836.

\bibitem{evans68}
{\sc D.~J. Evans}, {\em The use of pre-conditioning in iterative methods for
  solving linear equations with symmetric positive definite matrices}, J. Inst.
  Math. Appl., 4 (1968), pp.~295--314.

\bibitem{evans73}
\leavevmode\vrule height 2pt depth -1.6pt width 23pt, {\em The analysis and
  application of sparse matrix algorithms in the finite element method}, in The
  Mathematics of Finite Elements and Applications, Academic Press, London,
  1973, pp.~427--447.

\bibitem{FGS96}
{\sc A.~Forsgren, P.~E. Gill, and J.~R. Shinnerl}, {\em Stability of symmetric
  ill-conditioned systems arising in interior methods for constrained
  optimization}, SIAM J. Matrix Anal. Appl., 17 (1996), pp.~187--211.

\bibitem{goluboleary}
{\sc G.~H. Golub and D.~P. O'Leary}, {\em Some history of the conjugate
  gradient and {L}anczos algorithms: 1948--1976}, SIAM Rev., 31 (1989),
  pp.~50--102.

\bibitem{golubvanloan}
{\sc G.~H. Golub and C.~F. Van~Loan}, {\em Matrix computations}, vol.~3 of
  Johns Hopkins Series in the Mathematical Sciences, Johns Hopkins University
  Press, Baltimore, MD, 1983.

\bibitem{Gutknecht90}
{\sc M.~H. Gutknecht}, {\em The unsymmetric {L}anczos algorithms and their
  relations to {P}ade approximation, continued fractions, and the {QD}
  algorithm}.
\newblock Preliminary Proceedings of the Copper Mountain Conference on
  Iterative methods, April, 1990.

\bibitem{Gutknecht92}
\leavevmode\vrule height 2pt depth -1.6pt width 23pt, {\em A completed theory
  of the unsymmetric {L}anczos process and related algorithms, part {I}}, SIAM
  J. Matrix Anal. Appl., 13 (1992), pp.~594--639.

\bibitem{gutknechtacta}
\leavevmode\vrule height 2pt depth -1.6pt width 23pt, {\em Lanczos-type solvers
  for nonsymmetric linear systems of equations}, in Acta numerica, 1997, vol.~6
  of Acta Numer., Cambridge Univ. Press, Cambridge, 1997, pp.~271--397.

\bibitem{Gutknecht07}
\leavevmode\vrule height 2pt depth -1.6pt width 23pt, {\em {A brief
  introduction to Krylov space methods for solving linear systems}}, in
  {Frontiers of Computational Science}, {Springer-Verlag Berlin}, {2007},
  pp.~{53--62}.
\newblock {International Symposium on Frontiers of Computational Science,
  Nagoya Univ, Nagoya, Japan, Dec 12-13, 2005}.

\bibitem{Gutknecht00}
{\sc M.~H. Gutknecht and K.~J. Ressel}, {\em Look-ahead procedures for
  {L}anczos-type product methods based on three-term {L}anczos recurrences},
  SIAM J. Matrix Anal. Appl., 21 (2000), pp.~1051--1078.

\bibitem{hanke}
{\sc M.~Hanke}, {\em Conjugate gradient type methods for ill-posed problems},
  vol.~327 of Pitman Research Notes in Mathematics Series, Longman Scientific
  \& Technical, Harlow, 1995.

\bibitem{HestenesStiefel}
{\sc M.~R. Hestenes and E.~Stiefel}, {\em Methods of conjugate gradients for
  solving linear systems}, J. Research Nat. Bur. Standards, 49 (1952),
  pp.~409--436 (1953).

\bibitem{lanczos50}
{\sc C.~Lanczos}, {\em An iteration method for the solution of the eigenvalue
  problem of linear differential and integral operators}, J. Research Nat. Bur.
  Standards, 45 (1950), pp.~255--282.

\bibitem{lanczos52}
\leavevmode\vrule height 2pt depth -1.6pt width 23pt, {\em Solution of systems
  of linear equations by minimized-iterations}, J. Research Nat. Bur.
  Standards, 49 (1952), pp.~33--53.

\bibitem{luenberger}
{\sc D.~G. Luenberger}, {\em Linear and nonlinear programming}, Addison-Wesley
  Pub Co, Boston, MA, second~ed., 1984.

\bibitem{meurant}
{\sc G.~Meurant and Z.~Strako{\v{s}}}, {\em The {L}anczos and conjugate
  gradient algorithms in finite precision arithmetic}, Acta Numer., 15 (2006),
  pp.~471--542.

\bibitem{paige76}
{\sc C.~C. Paige}, {\em Error analysis of the {L}anczos algorithm for
  tridiagonalizing a symmetric matrix}, J. Inst. Math. Appl., 18 (1976),
  pp.~341--349.

\bibitem{paige94}
\leavevmode\vrule height 2pt depth -1.6pt width 23pt, {\em Krylov subspace
  processes, {K}rylov subspace methods, and iteration polynomials}, in
  Proceedings of the {C}ornelius {L}anczos {I}nternational {C}entenary
  {C}onference ({R}aleigh, {NC}, 1993), Philadelphia, PA, 1994, SIAM,
  pp.~83--92.

\bibitem{paigesaunders75}
{\sc C.~C. Paige and M.~A. Saunders}, {\em Solutions of sparse indefinite
  systems of linear equations}, SIAM J. Numer. Anal., 12 (1975), pp.~617--629.

\bibitem{reid71}
{\sc J.~K. Reid}, {\em On the method of conjugate gradients for the solution of
  large sparse systems of linear equations}, in Large sparse sets of linear
  equations ({P}roc. {C}onf., {S}t. {C}atherine's {C}oll., {O}xford, 1970),
  Academic Press, London, 1971, pp.~231--254.

\bibitem{saad}
{\sc Y.~Saad}, {\em Iterative methods for sparse linear systems}, Society for
  Industrial and Applied Mathematics, Philadelphia, PA, second~ed., 2003.

\bibitem{cgwopain}
{\sc J.~R. Shewchuk}, {\em An introduction to the conjugate gradient method
  without the agonizing pain}, tech. rep., Carnegie-Mellon University,
  Pittsburgh, PA, USA, 1994.

\bibitem{stiefel55}
{\sc E.~Stiefel}, {\em Relaxationsmethoden bester {S}trategie zur {L}\"osung
  linearer {G}leichungssysteme}, Comment. Math. Helv., 29 (1955), pp.~157--179.

\end{thebibliography}
\bibliographystyle{siam}

\end{document}